\documentclass[11pt, reqno, draft]{amsart}

\usepackage{mathrsfs}
\usepackage{amsfonts}
\usepackage{amsthm}
\usepackage{amsmath}
\usepackage{graphicx}
\usepackage{amscd,amssymb,amsthm, color}
\usepackage{setspace}

\numberwithin{equation}{section}

\newtheorem{theorem}{Theorem}

\newtheorem{remark}{Remark}

\begin{document}

\allowdisplaybreaks

\thispagestyle{plain}

\vspace{4cc}
\begin{center}

{\LARGE \bf On $p$--extended Mathieu series}
\rule{0mm}{6mm}\renewcommand{\thefootnote}{} \vspace{1cc}

{\large \bf Tibor K. Pog\'any$^\dag$ and Rakesh K. Parmar$^\ddag$}
\bigskip

{\it $^\dag$Faculty od Maritime Studies, University of Rijeka, Studentska 2, 51000 Rijeka, Croatia \and 
Institute of Applied Mathematics, \'Obuda University, 1034 Budapest, Hungary \\

$^\ddag$Department of Mathematics, Government College of Engineering and Technology,
Bikaner 334004, Rajasthan State, India}

\footnotetext{E-mail: \texttt{poganj@pfri.hr} (T. K. Pog\'any),\texttt{rakeshparmar27@gmail.com} (R. K. Parmar)}

\vspace{2cc}

\parbox{24cc}
{\small {\bf Abstract}. Motivated by several generalizations of the well--known Mathieu series, the main object of this paper is
to introduce new extension of generalized  Mathieu series and to derive various integral representations of such series. Finally
master bounding inequality is established using the newly derived integral expression.

\bigskip

\bigskip

}
\end{center}

\section{Introduction and motivation}

The series of the form
   \[ S(r) = \sum_{n \geq 1}\dfrac{2n}{(n^2+r^2)^2},\qquad r>0 \]
is known in literature as Mathieu series. \'Emile Leonard Mathieu was the first who investigated such series in 1890 in his book
\cite{Mathieu}. A remarkable useful integral representation for $S(r)$  is given by Emersleben \cite{Emer} in the following
elegant form
   \[ S(r) = \frac1r \int_0^\infty \frac{x \sin(rx)}{{\rm e}^x-1}\,{\rm d}x\,, \]
which can also be written in terms of the Riemann Zeta function $\zeta(s) = \sum_{n \ge 1} n^{-s}, s>1$ as
\cite[p. 863, Eq. (2.3)]{CS}( with $n$ by $n+1$ )
   \begin{equation}  \label{A12}
      S(r) = 2\sum_{n \geq 0}  (-1)^{n}\;(n+1)\;  \zeta(2n+3)\;r^{2n}, \qquad |r|<1.
   \end{equation}
The so--called generalized Mathieu series with a fractional power reads \cite[p. 2, Eq. (1.6)]{CL} (also see 
\cite[p. 181]{GP})
   \begin{equation}\label{A2}
      S_\mu(r) = \sum_{n \geq 1}\dfrac{2n}{(n^2+r^2)^{\mu+1}}, \qquad r>0,\,\mu>0;
   \end{equation}
such series has been widely considered in mathematical literature  (see e.g. papers by Diananda \cite{Dia}, Cerone and Lenard
\cite{CL} and Pog\'any {\em et al.} \cite{PST}). Cerone and Lenard also gave a series representation of $S_\mu(r)$  in terms of the Riemann Zeta function \cite[p. 3, Eq. (2.1)]{CL}
   \begin{equation}\label{A21}
      S_{\mu}(r) = 2\sum_{n \geq 0} (-1)^n \, \left({\mu+n} \atop n\right)\zeta(2\mu+2n+1)\, r^{2n}, \qquad |r|<1\,;
	 \end{equation}
in \cite{CL} was not mentioned the convergence region $|r|<1$. To show \eqref{A21} it is enough to expand the summands
in \eqref{A2} into a binomial series for $r \in (-1, 1)$ (compare \cite[p. 72, Proposition 1]{PST}). Cerone and Lenard derived
also the next integral expression \cite[p. 3, Theorem 2.1]{CL} (also consult \cite[p. 181, Eq. (1.3)]{GP})
   \begin{equation}\label{A22}
      S_{\mu}(r) = \frac{\sqrt{\pi}}{(2r)^{\mu-\frac12}\Gamma(\mu+1)} \int_0^\infty
                   \frac{x^{\mu + \frac12}}{{\rm e}^x-1}J_{\mu-\frac12}(rx)\,{\rm d}x, \qquad \mu>0.
   \end{equation}
Motivated by the previous extension and by huge spectrum of other generalizations of the Mathieu series, the main aim of
this paper is to study certain another types of, in new fashion generalized, Mathieu series.

Having in mind \eqref{A21} let the {\it $p$--extended Mathieu series} be defined as			
	 \begin{equation}\label{A3}
	    S_{\mu,p}(r) = 2\sum_{n \ge 0} r^{2n} (-1)^{n}\binom{\mu+n}{n} \zeta_{p}(2\mu+2n+1),
	 \end{equation}
where $p\geq0$; $\mu>0$, \, $|r|<1$ and $\zeta_{p}$ stands for the $p$--extended Riemann zeta function \cite{Ch-Qa-Bo-Ra-Zu}
   \begin{equation} \label{ExZeta}
      \zeta_{p}(\alpha) = \frac1{\Gamma(\alpha)}\int_{0}^{\infty} \frac{x^{\alpha-1}\, {\rm e}^{-\frac{p}x}}{{\rm e}^x-1} \,{\rm d}x
   \end{equation}
defined for $\Re(p)>0$ or $p=0$ and $\Re(\alpha)>0$, which reduces to Riemann zeta function when $p = 0$. It is also important 
to quote that \eqref{A3} reduces to \eqref{A21} when $p = 0$, while taking $\mu=1$ we yield \eqref{A12}.

\section{Integral representation, transforms and  series representations of $S_{\mu,p}(r)$} \label{sec2}

In this section we derive an integral expression for the $p$--extended Mathieu series $S_{\mu,p}(r)$. Then its various Mellin and
Laplace transforms are exposed.

\begin{theorem}\label{tm1} For all $\Re(p)>0$ or $p=0$, $\mu>0$ and $r>0$ the following integral representation for the extended generalized
 Matheiu series $S_{\mu,p}(r)$ holds true:
   \begin{equation}\label{A5}
      S_{\mu,p}(r) = \frac{\sqrt{\pi}}{(2r)^{\mu-\frac12}\Gamma(\mu+1)}\int_0^\infty \frac{x^{\mu+\frac12} {\rm e}^{-\frac{p}x}}
			               {{\rm e}^x-1} J_{\mu-\frac12}(rx)\,{\rm d}x\,.
   \end{equation}
\end{theorem}

\begin{proof} Using the series representation of $J_\nu$ \cite[p. 40]{Wat}
   \[ J_\nu(x) = \sum_{n \geq 0}\dfrac{(-1)^n}{n!\,\Gamma(n+\nu+1)}\left(\dfrac x2\right)^{2n+\nu}, \]
valid for all $\nu,x\in\mathbb{C}$ we can simplify an integral given in \eqref{A5} as:
   \begin{align*}
      \mathscr I &= \int_0^\infty \frac{x^{\mu+\frac12}{\rm e}^{-\frac{p}x}}{e^t-1} J_{\mu-\frac12}(rx) \;{\rm d}x
			            = \sum_{n \geq 0}\frac{(-1)^{n} (\frac{r}{2})^{\mu+2n-\frac12}}{n! \Gamma(\mu + n+\frac{1}{2})}
				            \int_0^\infty \frac{x^{2\mu+2n} {\rm e}^{-\frac{p}x} }{{\rm e}^x-1} \;{\rm d}x\\
				         &= \sum_{n \geq 0}\frac{(-1)^{n} (\frac{r}{2})^{\mu+2n-\frac12}}{n! \Gamma(\mu +n+\frac12)}
				            \,\Gamma(2\mu+2n+1) \,\zeta_p(2\mu+2n+1),
   \end{align*}
where in the last equality the definition of the $p$-extended Riemann Zeta function \eqref{ExZeta} was used.

Finally, with the help of duplication formula $\sqrt{\pi}\,\Gamma(2z)=2^{2z-1}\Gamma(z)\Gamma(z+\frac{1}{2})$
we get
   \[ \mathscr I = \frac{2\,(2r)^{\mu-\frac{1}{2}}\Gamma(\mu+1)}{\sqrt{\pi}}\sum_{n \geq 0}r^{2n} (-1)^{n}\binom{\mu+n}{n}
	                 \zeta_{p}(2\mu+2n+1) \]
which leads to the desired result.
\end{proof}

\begin{remark}
The integral expression \eqref{A5} one reduces to \eqref{A22} when $p=0$.
\end{remark}

In what follows we derive Mellin and Laplace transforms of the newly constructed series $S_{\mu,p}(r)$.

The Mellin and Laplace transforms (respectively) of some suitably integrable function $f$ with index $s$ are defined by
   \[ \mathscr M_x\{f(x)\}(s) = \int_0^\infty x^{s-1}\; f(x)\; {\rm d}x, \qquad
      \mathscr L_x \{f (x)\}(s) = \int_0^\infty {\rm e}^{-s x} \,f (x)\, {\rm{d}}x. \]
provided that the corresponding integrals exist.

\begin{theorem} The Mellin transform of the extended generalized Matheiu series $S_{\mu,p}(r)$ read as follows:
   \[ \mathscr M_p \{S_{\mu,p}(r)\}(s) = 2s\,\Gamma^2(s) \sum_{n\ge0} (-r^2)^n \binom{\mu+n}{n}
				                                       \binom{2\mu+2n+s}{2\mu+2n} \zeta(2\mu+2n+s+1), \]
in the range $|r|<1$. Moreover, for $\mu>0$;\,$\,0<\Re(s)<\mu+1$ and $\Re(p)>0$,
   \begin{equation} \label{M2}
         \mathscr M_r \{S_{\mu,p}(r)\}(s)  = {\rm B}\left(\frac{s}{2},\mu+1-\frac{s}{2}\right)\,\zeta_{p}(2\mu-s+1),
   \end{equation}
where ${\rm B}(x, y), \min\{\Re(x), \Re(y)\}>0$ stands for the Euler Beta function.
\end{theorem}

\begin{proof} Using the definition of the Mellin transform, we find from \eqref{A3}
   \begin{align*}
      \mathscr M_p \left\{S_{\mu,p}(r)\right\}(s) &=  \int_0^{\infty} p^{s-1}S_{\mu,p}(r)\, {\rm d}p \\
         &=  2\sum_{n\ge0} r^{2n} (-1)^{n}\binom{\mu+n}{n}\int_0^{\infty} p^{s-1}\zeta_{p}(2\mu+2n+1)\, {\rm d}p\\
         &= 2\Gamma(s)\Gamma(s+1)\sum_{n\ge0} r^{2n} (-1)^{n}\binom{\mu+n}{n} \binom{2\mu+2n+s}{2\mu+2n} \zeta(2\mu+2n+s+1),
   \end{align*}
where in the last equality we used  the formula \cite[p. 1244, Eq. (3.6)]{Ch-Qa-Bo-Ra-Zu}
   \[ \int_0^{\infty} p^{s-1}\zeta_{p}(\alpha)\, {\rm d}p =
              \frac{\Gamma(s)\Gamma(\alpha+s)}{\Gamma(\alpha)}\zeta(\alpha+s), \qquad \Re(\alpha)>0,\,\, \Re(s)>0. \medskip \]
Next, with the help of the Weber--Sonine integral \cite[p. 391, Eq. 13.24(1)]{Wat}
   \[ \int_0^{\infty} x^{\mu-\nu-1}J_{\nu}(x)\, {\rm d}x =
              \frac{\Gamma\left(\frac{\mu}2\right)}{2^{\nu-\mu+1}\Gamma\left(1+\nu-\frac{\mu}2\right)},
							\qquad 0<\Re(\mu)<\Re(\nu)+\frac12, \]
the integral representation \eqref{A5} derived in Theorem~\ref{tm1} and the definition of extended Riemman Zeta $\zeta_p$, we find that
   \begin{align*}
      \mathscr M_r \left\{S_{\mu,p}(r)\right\}(s) &=  \int_0^\infty r^{s-1}S_{\mu,p}(r)\, {\rm d}r \\
         &= \frac{\sqrt{\pi}}{2^{\mu-\frac12}\Gamma(\mu+1)}\int_0^\infty
				    \frac{x^{\mu+\frac12} {\rm e}^{-\frac{p}x} }{{\rm e}^x - 1} \left(\int_0^\infty r^{s-\mu-\frac12}
						J_{\mu-\frac{1}{2}}(rx)\,{\rm d}r\right) \;{\rm d}x\\
         &= \frac{\sqrt{\pi}\,\Gamma\left(\frac{s}2\right)}{2^{2\mu-s}\Gamma(\mu+1)\,\Gamma\left(\mu+\frac{1-s}2\right)}
				    \int_0^\infty \frac{x^{2\mu-s} {\rm e}^{-\frac{p}x} }{{\rm e}^x - 1}\;{\rm d}x\\
         &= \frac{\sqrt{\pi}\,\Gamma\left(\frac{s}2\right)\,\Gamma(2\mu-s+1)\,\zeta_{p}(2\mu-s+1)}{2^{2\mu-s}\Gamma(\mu+1)\,
				    \Gamma\left(\mu+\frac{1-s}2\right)},
   \end{align*}
which gives \eqref{M2} with the help of the duplication formula for the Gamma function and the relation between
Beta and Gamma functions.
\end{proof}

\begin{theorem}
For the $p$--extended Mathieu series $S_{\mu,p}(r)$ we have the Laplace transform formula
   \begin{equation} \label{L1}
      \mathscr L_r\{S_{\mu,p}(r)\}(x) = \frac{2}{x\Gamma(2\mu+1)} \int_{0}^{\infty} \frac{t^{2\mu} {\rm e}^{-\frac{p}{t}}}{{\rm e}^t-1}\,
			             {}_2F_1 \Bigg[ \begin{array}{c} 1, \frac12\\ \mu+\frac12\end{array} \Bigg|  -\frac{ t^2}{x^2} \Bigg] \;{\rm{d}}t.
   \end{equation}
   provided that the each member of \eqref{L1} exists.
\end{theorem}

\begin{proof}
Using the Laplace transform formula \cite[p. 49, Eq. 7.7.3(16)]{EM}
   \[ \mathscr L_x\{ x^{\lambda-1}J_\nu(\rho x)\}(s) = \left(\dfrac{\rho}{2s}\right)^\nu
			                s^{-\lambda}\dfrac{\Gamma(\nu+\lambda)}{\Gamma(\nu+1)}\,
			                {}_2F_1 \Bigg[ \begin{array}{c} \frac12(\nu+\lambda),\frac12(\nu+\lambda+1)\\ \nu+1\, \end{array} \Bigg|
											-\frac{\rho^2}{s^2} \Bigg], \]
valid for all $|\Re(s)|> |\Im(\rho)|,\,\Re(\nu+\lambda)>0$ and the integral representation \eqref{A5}, we get
   \begin{align*}
      \mathscr L_r \left\{S_{\mu,p}(r)\right\}(x) &=  \int_0^{\infty} {\rm e}^{-xr}S_{\mu,p}(r)\, {\rm d}r \\
              &=  \frac{\sqrt{\pi}}{2^{\mu-\frac12}\Gamma(\mu+1)}\int_0^\infty 
							    \frac{t^{\mu+\frac12}{\rm e}^{-\frac{p}{t}}}{{\rm e}^t-1}\left(\int_{0}^{\infty}{\rm e}^{-xr}r^{\frac12-\mu}
									J_{\mu-\frac{1}{2}}(rt)\,{\rm d}r\right) \,{\rm d}t\\
              &=  \frac{\sqrt{\pi}}{2^{2\mu-1}x\Gamma(\mu+1)\Gamma(\mu+\frac{1}{2})}\int_{0}^{\infty}
							    \frac{t^{2\mu}{\rm e}^{-\frac{p}{t}}}{{\rm e}^t-1}\; {}_2F_1 \Bigg[ \begin{array}{c} 1, \frac12\\
									\mu+\frac12\end{array} \Bigg|  -\frac{ t^2}{x^2} \Bigg] \,{\rm d}t,
   \end{align*}
which becomes \eqref{L1} with the help of duplication formula.
\end{proof}

\begin{remark} 
It is interesting to note when $p=0$, \eqref{M2} and \eqref{L1} reduce to known results in \cite{EST}.
\end{remark}

\section{Master bounding inequality upon $S_{\mu,p}(r)$}

A set of bounding inequalities exist for the generalized Mathieu series $S_\mu(r)$; as their main surce we can list the articles
\cite{PST, Pogany1, Tomovski1}. To give upper bounds for $S_{\mu,p}(r)$ {\it via} $S_\mu(r)$, since the oscillatory behavior
of $J_\mu$ in the integrand of the integral representation \eqref{A22},  we are forced to consider the modulus of the input series.

Observing the $p$--kernel ${\rm e}^{-\frac px}, x>0$ introduced by Chaudhary {\em et al.} \cite{Ch-Qa-Bo-Ra-Zu}, instead of the
obvious bound ${\rm e}^{-\frac px} \leq 1$, for non--negative parameter $p$ we infer the more precise estimate
   \[ {\rm e}^{-\frac px} \leq \mathscr C_p(x) = \begin{cases}
			                \dfrac{2x}{p\,{\rm e}^2}                     & x \in \left(0, \dfrac p2\right) \\
										  \dfrac{4x}{p\,{\rm e}^2}-\dfrac1{{\rm e}^2}  & x \in \left[\dfrac p2, \dfrac{p}4 (1+{\rm e}^2) \right)\\
											1                                            & x \geq \dfrac{p}4 (1+{\rm e}^2)
								   \end{cases} \,,\qquad x>0\,. \]
Indeed, being $I(p/2, e^{-2})$ the inflection point in which the kernel is changing behavior from convex into concave in growing $x$,
the secant line joining the origin and $I$ is above the kernel's arc, while the tangent line in $I$ bounds the kernel from above in
the middle interval. The structure of $\mathscr C_p(x)$ {\it mutatis mutandis} splits the integration domain in \eqref{A5} getting
   \begin{align} \label{BU2}
	    |S_{\mu, p}(r)| &\leq \frac{\sqrt{\pi}}{(2r)^{\mu-\frac12}\Gamma(\mu+1)} \int_0^\infty \frac{x^{\mu+\frac12} \,
			                      \mathscr C_p(x)}{{\rm e}^x-1} |J_{\mu-\frac12}(rx)|\,{\rm d}x \nonumber \\
											&=    \frac{2\sqrt{\pi}\,(2r)^{\frac12-\mu}}{p\,{\rm e}^2\,\Gamma(\mu+1)}
														\int_0^{\frac{p}2}\, \frac{x^{\mu+\frac32}\,|J_{\mu-\frac12}(rx)|}{{\rm e}^x-1} \,{\rm d}x \nonumber \\
                      &\qquad + \frac{4\sqrt{\pi}\,(2r)^{\frac12-\mu}}{p\,{\rm e}^2\,\Gamma(\mu+1)}
														\int_{\frac{p}2}^{\frac{p}4 (1+{\rm e}^2)}\, \frac{x^{\mu+\frac12}\,
														|J_{\mu-\frac12}(rx)|}{{\rm e}^x-1} \Big(x-\frac{p}4\Big)\,{\rm d}x \nonumber \\
											&\qquad + \frac{\sqrt{\pi}}{(2r)^{\mu-\frac12}\Gamma(\mu+1)} \int_{\frac{p}4 (1+{\rm e}^2)}^\infty
											      \frac{x^{\mu+\frac12} \, }{{\rm e}^x-1} |J_{\mu-\frac12}(rx)|\,{\rm d}x\,.
	 \end{align}
All three integrals in \eqref{BU2} contain the same fashion integrands, but for $p>0$ we couldn't express these integrals in terms
{\it via} a finite linear combination of elementary and/or higher transcendental, special functions. That is the reason why we
list here various uniform and functional upper bounds upon $|J_\nu(x)|$, preferably the ones of polynomial decay compare for instance
\cite{baricz0} and \cite[Subsection 3.2]{RK}. For the sake of simplicity we recall here only few of them. Firstly, we mention
von Lommel's uniform bounds \cite{Lommel1}, \cite[pp. 548--549]{Lommel2}, \cite[p. 406]{Wat}:
   \[ |J_\nu(x)| \le 1,\qquad |J_{\nu+1}(x)| \le \frac1{\sqrt 2}, \qquad  \nu>0 ,\, x\in \mathbb R, \]
and the bound by Minakshisundaram and Sz\'asz \cite[p. 37, Corollary]{MinSzasz}
   \[ |J_\nu(x)| \le \dfrac{|x\,|^\nu}{2^\nu\,\Gamma(\nu+1)}, \qquad \nu \geq 0,\, x \in \mathbb R. \]
Further estimates were given by Landau \cite{Landau} with respect to $\nu$ and $x$ which are in a sense best possible
(outside of Bessel function's transition region)
   \begin{align} \label{H6}
      |J_{\nu}(x)| &\le b_L\, \nu^{-1/3}, \qquad b_L = \sqrt[3]{2}\sup_{x \geq 0} {\rm Ai}(t), \\ \label{H7}
      |J_{\nu}(x)| &\le c_L\, |x|^{-1/3}, \qquad c_L = \sup_{x \geq 0} {x^{1/3}J_0(x)}\, ;
   \end{align}
here $\text{Ai}(\cdot)$ denotes Airy function. In turn, Olenko answered to this challenge by \cite[Theorem 1]{Olenko}
   \begin{equation} \label{H8}
      \sup_{x\geq 0}\sqrt{x}\,|J_\nu(x)| \le b_L \sqrt{\nu^{1/3} + \frac{\alpha_1}{\nu^{1/3}}
                                          + \frac{3\alpha_1^2}{10\nu}}  =  d_O, \qquad \nu>0\, .
   \end{equation}
Here $\alpha_1$ is the smallest positive zero of $\text{Ai}$, being $b_L$ the first Landau's constant.
Further considerable upper bounds are listed e.g. in the works by Baricz {\em et al.} \cite{baricz0} and by Srivastava and Pog\'any
\cite{SrivPogany}. Also one draws the reader's attention to the sophisticated functional bounds by Krasikov \cite{Kras0, Kras1}.
We cover all these cases with a generic bound $|J_{\mu-\frac12}(x)| \leq C \cdot x^q, x>0$, where both, the absolute constant
$C = C(\mu)$ and the power $q$ are changing with the different kind bounds pointing out that the application of estimates mentioned,
and by the sake of simplicity not used in evaluating $S_{\mu, p}(r)$, we plane at another address.

At this point we establish the master inequality by virtue of the newly established integral expression \eqref{A5} covering all
above listed cases of Bessel function bounding inequalities.

\begin{theorem}
For all $p \geq 0, \mu>0; q> -\mu- \tfrac12$ and for all $r>0$ there holds
   \begin{align} \label{BU0}
	    |S_{\mu, p}(r)| &\leq \frac{C \sqrt{\pi}\,p^{\mu+q+\frac12}}{{\rm e}^2\,2^{2\mu+q}r^{\mu-\frac12}\,\Gamma(\mu+1)}
			                      \Bigg\{ \frac1{2(\mu+q)+5}\,\Big( \frac4{2(\mu+q)+3}+  \frac{p}{{\rm e}^{\frac p2}-1} \Big)  \nonumber \\
										  &\qquad + \frac{2K_1\,p}{2(\mu+q)+5} \Big( \Big(\frac{1+{\rm e}^2}2\Big)^{\mu+q+\frac52} - 1\Big)
											       + \frac{4K_2}{2(\mu+q)+3} \Big( \Big(\frac{1+{\rm e}^2}2\Big)^{\mu+q+\frac32} - 1\Big) \nonumber \\
										  &\qquad - \frac{(1+{\rm e}^2)\,p}{2\, [2(\mu+q)+1]
											      ({\rm e}^{\frac{p}4 (1+{\rm e}^2)}-1)}\,
											      \Big( \Big( \frac{1+{\rm e}^2}2\Big)^{\mu+q+\frac12} - 1\Big) \nonumber \\
										  &\qquad - \dfrac{{\rm e}^2 p (1+{\rm e}^2)^{\mu+q+\frac32}}
														{2^{\mu+q+\frac32} [2(\mu+q)+1]\,({\rm e}^{\frac{p}4 (1+{\rm e}^2)}-1)} \Bigg\} \nonumber \\
											&\quad + \frac{C \sqrt{\pi}}{(2r)^{\mu-\frac12}\,\Gamma(\mu+1)}\,
											      \Gamma\Big(\mu+q+\frac32\Big)\,\zeta\Big(\mu+q+\frac32\Big)\, .
	 \end{align}

\end{theorem}

\begin{proof} Consider the auxiliary integral
   \[ \mathscr K(\alpha, a, b) = \int_a^b \dfrac{x^{\alpha-1}}{{\rm e}^x-1}\, {\rm d}x; \qquad \alpha>1; 0 \leq a< b< \infty\, . \]
Being the function $x \mapsto x({\rm e}^x-1)^{-1}$ monotone decreasing and convex for $x>0$ we estimate this function's arc from
above with secant line crossing $A(a, a({\rm e}^a-1)^{-1})$ and $B(b, b({\rm e}^b-1)^{-1})$. Further, taking the lower bound
$x({\rm e}^x-1)^{-1} \geq b({\rm e}^b-1)^{-1}$ on the whole $(a, b)$ we achieve
   \begin{equation} \label{BU4}
	   \dfrac{b(b^{\alpha-1}-a^{\alpha-1})}{(\alpha-1)\,({\rm e}^b-1)} \leq \mathscr K(\alpha, a, b)
	                \leq \frac{K_1}\alpha\,\big(b^\alpha - a^\alpha\big)
									+    \frac{K_2}{\alpha-1}\,\big(b^{\alpha-1} - a^{\alpha-1} \big) \,, \qquad \alpha>1\,,
	 \end{equation}
where
   \[ K_1 = \Big( \frac{b}{{\rm e}^b-1} - \frac{a}{{\rm e}^a-1}\Big)\,\frac1{b-a};\qquad
	    K_2 = \Big( \frac1{{\rm e}^b-1} - \frac1{{\rm e}^a-1}\Big)\,\frac{ab}{b-a}\,.\]
Letting here $a \to 0+$, \eqref{BU4} one reduces to
   \begin{equation} \label{BU5}
	    \dfrac{b^\alpha}{(\alpha-1)\,({\rm e}^b-1)} \leq \mathscr K(\alpha, 0, b)
	                \leq \frac{b^{\alpha-1}}\alpha\,\Big( \frac1{\alpha-1}+  \frac{b}{{\rm e}^b-1} \Big) \,, \qquad \alpha>1\,.
	 \end{equation}
On the other hand, we get
   \begin{align} \label{BU6}
	    \overline{\mathscr K}(\alpha, b) &= \int_b^\infty \dfrac{x^{\alpha-1}}{{\rm e}^x-1}\, {\rm d}x
			                    = \int_0^\infty \dfrac{x^{\alpha-1}}{{\rm e}^x-1}\, {\rm d}x - \mathscr K(\alpha, 0, b) \nonumber \\
	                       &\leq \Gamma(\alpha)\, \zeta(\alpha) - \dfrac{b^\alpha}{(\alpha-1)\,({\rm e}^b-1)}\,,
	 \end{align}
where for all $b>1$ the right--hand--side estimate is not redundant, namely in this $b$--domain the upper bound is strict positive.

In the introductory part of this section we list diverse bounding inequalities for the Bessel function of the first kind of positive
argument. Bearing in mind \eqref{BU2} we conclude
   \begin{align*}
	    |S_{\mu, p}(r)| &\leq \frac{2C \sqrt{\pi}}{p\,{\rm e}^2\,(2r)^{\mu-\frac12}\,\Gamma(\mu+1)}
			                      \Big\{ \mathscr K\Big(\mu+q+\frac52, 0, \frac p2\Big)
                       + 2\, \mathscr K\Big(\mu+q+\frac52,\frac{p}2, \frac{p}4 (1+{\rm e}^2)\Big) \nonumber \\
											&\qquad - \frac{p}2\, \mathscr K\Big(\mu+q+\frac32,\frac{p}2, \frac{p}4 (1+{\rm e}^2)\Big)
											 + \frac{p\,{\rm e}^2}2\, \overline{\mathscr K} \Big(\mu+q+\frac32,\frac{p}4 (1+{\rm e}^2)\Big) \Big\} \,.
	 \end{align*}
Applying the estimates \eqref{BU4}, \eqref{BU5} upon $\mathscr K$ and \eqref{BU6} for $\overline{\mathscr K}$, it follows
   \begin{align*}
	    |S_{\mu, p}(r)| &\leq \frac{2C \sqrt{\pi}}{{\rm e}^2\,(2r)^{\mu-\frac12}\,\Gamma(\mu+1)}
			                      \Bigg\{ \frac{p^{\mu+q+\frac12}}{2^{\mu+q+\frac32} [2(\mu+q)+5]}\,
													  \Big( \frac4{2(\mu+q)+3}+  \frac{p}{{\rm e}^{\frac p2}-1} \Big)  \nonumber \\
										  &\qquad + \frac{K_1\,p^{\mu+q+\frac32}}{2^{\mu+q+\frac12}[2(\mu+q)+5]}
										        \Big( \Big(\frac{1+{\rm e}^2}2\Big)^{\mu+q+\frac52} - 1\Big) \nonumber \\
											&\qquad + \frac{K_2\,p^{\mu+q+\frac12}}{2^{\mu+q-\frac12}[2(\mu+q)+3]}
										        \Big( \Big(\frac{1+{\rm e}^2}2\Big)^{\mu+q+\frac32} - 1\Big) \nonumber \\
										  &\qquad - \frac{(1+{\rm e}^2)\,p^{\mu+q+\frac32}}{2^{\mu+q+\frac52} [2(\mu+q)+1]
											      ({\rm e}^{\frac{p}4 (1+{\rm e}^2)}-1)}\,
											      \Big( \Big( \frac{1+{\rm e}^2}2\Big)^{\mu+q+\frac12} - 1\Big) \nonumber \\
										  &\qquad + \frac{{\rm e}^2}2\, \Gamma\Big(\mu+q+\frac32\Big)\,\zeta\Big(\mu+q+\frac32\Big) \notag \\
											&\qquad - \dfrac{{\rm e}^2 p^{\mu+q+\frac32} (1+{\rm e}^2)^{\mu+q+\frac32}}
														{2^{2(\mu+q)+3} [2(\mu+q)+1]\,({\rm e}^{\frac{p}4 (1+{\rm e}^2)}-1)} \Bigg\}\,.
	 \end{align*}
Now routine	steps lead to the assertion.
\end{proof}

The specific estimates upon $J_{\mu - \frac12}(x)$ in \eqref{BU0} form a set of respective particular bounds:\medskip

\noindent $\mathsf A.$ Taking $C = 2^{-\frac12}, q=0$, we infer a Lommel--type bound from \eqref{BU0} if $\mu>\tfrac12$. 

\noindent $\mathsf B.$ When $q = \mu-\frac12 \geq 0$ and respectively
   \[C(r) = \dfrac{r^{\mu-\frac12}}{2^{\mu-\frac12}\,\Gamma\left(\mu+\frac12\right)}, \]
we arrive at the Minakshisundaram--Sz\'asz--type bound, which surprisingly becomes $r$--invariant. \medskip

\noindent $\mathsf C.$ Making use of Landau's first estimate with $q=0$ and
   \[ C(r) = \dfrac{b_L}{\sqrt[3]{\mu-\tfrac12}}, \qquad \mu>\tfrac12, \]
where $b_L$ was defined in \eqref{H6}, we get a bound of the same magnitude (in $r$) then von Lommel's one which is now equal to 
$\mathscr O(r^{-\mu+\frac12})$. \medskip

\noindent $\mathsf D.$ Next, using Landau's second estimate \eqref{H7} with $q=-\tfrac13$ and
   \[ C(r) = \dfrac{c_L}{\sqrt[3]{r}}, \qquad \mu>\frac13 \]
increases the magnitude of \eqref{BU0} into $\mathscr O(r^{-\mu+\frac16})$, $c_L$ being described around \eqref{H7}. \medskip

\noindent $\mathsf E.$ Finally, putting $q = -\tfrac12$ and according to \eqref{H8}
   \[ C(r) = \frac{d_O}{\sqrt{r}}\, , \qquad \mu>\frac12, \]
implies the Olenko bound which magnitude reads $\mathscr O(r^{-\mu})$.

\end{document}